\documentclass[12pt]{amsart}
\usepackage{amsmath}
\usepackage{amssymb}
\usepackage{longtable}
\usepackage{tikz}
\usepackage{footnote}
\usepackage[all]{xy}
\usepackage{longtable}
\allowdisplaybreaks[1]
\newtheorem{theorem}[equation]{Theorem}
\newtheorem{lemma}[equation]{Lemma}
\newtheorem{proposition}[equation]{Proposition}
\newtheorem{corollary}[equation]{Corollary}
\newtheorem{conjecture}[equation]{Conjecture}
\newtheorem{claim}[]{Claim}

\theoremstyle{definition}

\theoremstyle{remark}

\numberwithin{equation}{section}

\DeclareMathAlphabet{\matheur}{U}{eur}{m}{n}

\DeclareMathOperator{\Gal}{Gal}

\DeclareMathOperator{\NS}{NS}
\DeclareMathOperator{\T}{T}
\DeclareMathOperator{\Km}{Km}

\newmuskip\pFqskip
\mathchardef\pFcomma=\mathcode`, % keep a copy of the comma

\newcommand*\pFq[5]{%
  \begingroup
  \begingroup\lccode`~=`,
    \lowercase{\endgroup\def~}{\pFcomma\mkern\pFqskip}%
  \mathcode`,=\string"8000
  {}_{#1}F_{#2}\biggl(\genfrac..{0pt}{}{#3}{#4};#5\biggr)%
  \endgroup
}
\renewcommand{\Im}{\operatorname{Im}}
\renewcommand{\Re}{\operatorname{Re}}

\begin{document}

\title[Mahler measures as linear combinations of $L$-values]{Mahler measures as linear combinations of $L$-values of multiple modular forms}

\author{Detchat Samart}
\address{Department of Mathematics, Texas A{\&}M University, College
Station,
TX 77843, USA} \email{detchats@math.tamu.edu}

\thanks{The author's research was partially supported by NSF Grant DMS-1200577}

%    General info

\date{\today}

\begin{abstract}
We study the Mahler measures of certain families of Laurent polynomials in two and three variables. Each of the known Mahler measure formulas for these families involves $L$-values of at most one newform and/or at most one quadratic character. In this paper, we show, either rigorously or numerically, that the Mahler measures of some polynomials are related to $L$-values of multiple newforms and quadratic characters simultaneously. The results suggest that the number of modular $L$-values appearing in the formulas significantly depends on the shape of the algebraic value of the parameter chosen for each polynomial. As a consequence, we also obtain new formulas relating special values of hypergeometric series evaluated at algebraic numbers to special values of $L$-functions.
\end{abstract}


\maketitle
\section{Introduction}\label{Sec:intro}
For any Laurent polynomial $P\in \mathbb{C}[X_1^{\pm 1},\ldots,X_n^{\pm 1}]$, the Mahler measure of $P$ is defined by 
$$m(P)=\int_0^1\cdots \int_0^1 \log |P(e^{2\pi i \theta_1},\ldots,e^{2\pi i \theta_n})|\,d\theta_1\cdots d\theta_n.$$
(In some parts of the literature, $m(P)$ is called \textit{the logarithmic Mahler measure of} $P$, but throughout this paper we shall omit the term \textit{logarithmic}.) In the univariate case, the Mahler measure can be calculated quite easily with the help of Jensen's formula. However, there seems not to be a general formula for Mahler measures of randomly chosen multivariate polynomials, and it is still unclear what are the precise ways that Mahler measures are related to the polynomials. It was Deninger \cite{Deninger} who first used the Bloch-Beilinson conjectures to predict that Mahler measures of certain polynomials are related to special values of $L$-functions. In particular, he conjectured that the following formula holds:
$$m(x+x^{-1}+y+y^{-1}+1)=\frac{15}{4\pi^2}L(E,2)=L'(E,0),$$ where $E$ is the elliptic curve of conductor $15$ defined by the projective closure of the zero locus of $x+x^{-1}+y+y^{-1}+1$. 
This formula had been conjectural for years before being proved by Rogers and Zudilin \cite{RZ}. 

To consider more general situations, we let $$P_k=x+x^{-1}+y+y^{-1}+k,$$ where $k\in\mathbb{C}.$ It was verified numerically by Boyd \cite{Boyd} that, for many integral values of $k\neq 0,\pm 4$, if $E_k$ is the elliptic curve over $\mathbb{Q}$ determined by the zero locus of $P_k$, then 
\begin{equation}\label{E:ck}
m(P_k)\stackrel{?}=c_kL'(E_k,0),
\end{equation}
where $c_k$ is a rational number of small height. (Here and throughout $\stackrel{?}=$ means that they are equal to at least 25 decimal places.) Note that by the modularity theorem the relation \eqref{E:ck} is equivalent to 
\begin{equation*}\label{E:ck2}
m(P_k)\stackrel{?}=c_kL'(h_k,0),
\end{equation*}
where $h_k$ is the newform of weight 2 associated to $E_k$. (In most situations, we will be dealing with $L$-values of newforms rather than those attached to algebraic varieties.)
Although Boyd's results seem to be highly accurate, rigorous proofs of these formulas are quite rare (see Table~\ref{tab:Pk} below). Inspired by these results, Rodriguez Villegas \cite{RV} proved that $m(P_k)$ can be expressed in terms of Eisenstein-Kronecker series, and for certain values of $k$ they turn out to be related to special values of $L$-series of elliptic curves with complex multiplication. For instance, he proved that 
\begin{align}
m(P_{4\sqrt{2}})&=L'(E_{4\sqrt{2}},0)=L'(f_{64},0),\label{E:RV1} \\ 
m\left(P_{\frac{4}{\sqrt{2}}}\right)&=L'\left(E_{\frac{4}{\sqrt{2}}},0\right)=L'(f_{32},0), \label{E:RV2}
\end{align}
where $f_{64}$ and $f_{32}$ are newforms of weight 2 and level $64$ and $32$, associated to the elliptic curves $E_{4\sqrt{2}}$ and $E_{\frac{4}{\sqrt{2}}}$, respectively. He also observed from his numerical data that the relation \eqref{E:ck} seems to hold for every sufficiently large $k$ such that $k^2\in \mathbb{Z}$. One of the possible reasons why one needs $k$ to be the square root of a rational number or an integer is that $E_k$ has a Weierstrass form
\begin{equation*}
y^2=x^3+\frac{k^2}{8}\left(\frac{k^2}{8}-1\right)x^2+\frac{k^4}{256}x,
\end{equation*} 
which is defined over $\mathbb{Q}$ if $k^2\in \mathbb{Q}$.
For a complete list of conjectured formulas obtained from Rodriguez Villegas's computational experiments, see \cite[Tab.~4]{RV}. 
\begin{savenotes}
\begin{table}[ht]
\centering
    \begin{tabular}{ | c | c |}
    \hline
    $k^2$ & Reference(s)  \\ \hline
    $8,16$\footnote{When $k=\pm 4$, $E_k$ is a curve of genus $0$ and $m\left(P_k\right)=2L'(\chi_{-4},-1)$, where $\chi_{-4}(n)=\left(\frac{-4}{n}\right).$}$,18,32$ & \cite{RV}  \\ \hline
    $1$ & \cite{RZ,Zudilin}  \\ \hline
    $4,64$ & \cite{LR}  \\ \hline
    $-4,-1,2$ & \cite{RZ0,Zudilin}  \\ \hline
    \end{tabular}
    \caption{Values of $k$ for which Formula~\eqref{E:ck} is known to be true.}
\label{tab:Pk}
\end{table}
\end{savenotes}

In Section~\ref{Sec:two} we will deduce formulas for $m(P_k)$ when $k=\sqrt{8\pm6\sqrt{2}}$. Indeed, we will prove that 
\begin{equation}\label{E:first}
m\left(P_{\sqrt{8\pm6\sqrt{2}}}\right)=\frac{1}{2}\left(L'(f_{64},0)\pm L'(f_{32},0)\right).
\end{equation}
Using similar arguments one obtains conjectured formulas in terms of two different $L$-values for $m(P_k)$ when $k=\sqrt{8\pm9\sqrt{2}}.$ Observe that in these cases $k^2\notin\mathbb{Q},$ so it is not surprising that our results are somewhat different from those of Rodriguez Villegas. In addition, we consider the Hesse family  
$$Q_k=x^3+y^3+1-kxy.$$ The corresponding elliptic curve defined by $Q_k$ has a Weierstrass model
\begin{equation*}
y^2=x^3-27k^6x^2+216k^9(k^3-27)x-432k^{12}(k^3-27)^2.
\end{equation*} 
This family was also investigated in \cite{RV}, and it was pointed out that the Mahler measures of $Q_k$ appear to be of the form \eqref{E:ck} when $k$ is sufficienly large and $k^3\in\mathbb{Z},$ as hinted by the Weierstrass form given above. On the other hand, we will prove that if $k=\sqrt[3]{6-6\sqrt[3]{2}+18\sqrt[3]{4}},$ then 
$$m(Q_{k})=\frac{1}{2}\left(L'(f_{108},0)+L'(f_{36},0)-3L'(f_{27},0)\right),$$ where $f_N$ is a newform of weight $2$ and level $N$. Remark that, in this case, the elliptic curve corresponding to $Q_k$ is defined over $\mathbb{Q}\left(\sqrt[3]{2}\right)$ rather than $\mathbb{Q}$.

In Section~\ref{Sec:three} we will establish some formulas concerning three-variable Mahler measures. The author showed in \cite{Samart} that 
for many values of $k$ the Mahler measures of the following Laurent polynomials:
\begin{align*}
&(x+x^{-1})(y+y^{-1})(z+z^{-1})+k,\\
&(x+x^{-1})^2(y+y^{-1})^2(1+z)^3z^{-2}-k,\\
&x^4+y^4+z^4+1+kxyz
\end{align*}
are of the form 
\begin{equation} \label{E:three}
m(P)=c_1L'(g,0)+c_2L'(\chi,-1)
\end{equation}
for some CM newform $g$ of weight $3$ with rational Fourier coefficients, an odd quadratic character $\chi$, and $c_1,c_2\in\mathbb{Q}.$ To obtain the formulas of type \eqref{E:three}, it seems that the chosen value of $k$ necessarily satisfies similar conditions as observed in the two-variable case. For instance, for the last family, $k$ must be sufficiently large and $k^4\in\mathbb{Z}.$ By simple transformation, one sees that the $K3$ surfaces corresponding to this family are birational to those defined by the zero loci of $x^4+y^4+z^4+xyz+k^{-4}$. When $k^4\in\mathbb{Q}$, the $K3$ surfaces are then defined over $\mathbb{Q}$ and are known to be modular in the sense that their attached $L$-series coincide with $L$-series of weight three cusp forms by a result of Livn\'{e} \cite{Liv}. Moreover, in the case of singular $K3$ surfaces, the corresponding cusp forms are CM, and a complete list of them can be found in \cite[Tab.~1]{Schutt}. Therefore, in this particular case, one might expect the Mahler measures to be related to CM weight three cusp forms. On the other hand, we will give first examples of Mahler measures of polynomials in this family when $k^4$ are algebraic integers but not rational integers which reveal similar phenomena as seen in the two-variable case. For example, it will be proved that when $k=\sqrt[4]{26856+15300\sqrt{3}}$ the following equality is true:
\begin{equation*}
m(x^4+y^4+z^4+1+kxyz)=\frac{5}{48}\left(20L'(g_{12},0)+4L'(g_{48},0)+11L'(\chi_{-3},-1)+8L'(\chi_{-4},-1)\right),
\end{equation*}
where $g_N$ is a newform of weight $3$ and level $N$ and $\chi_{D}(n)=\left(\frac{D}{n}\right).$ 

In Section~\ref{Sec:four} we give lists of values of $k$ corresponding to singular $K3$ surfaces in the families given by the zero loci of $(x+x^{-1})(y+y^{-1})(z+z^{-1})+k$ and $x^4+y^4+z^4+1+kxyz$. It turns out that the Mahler measures of the polynomials defining these singular $K3$ surfaces are all conjecturally equal to rational linear combinations of modular and Dirichlet $L$-values. The (conjectural) Mahler measure formulas obtained from numerical computations are illustrated in Table~\ref{tab:f2}-\ref{tab:f4} at the end of this paper.

In Section~\ref{Sec:five} we establish a functional equation of three-variable Mahler measures, which gives us a five-term relation between the Mahler measures with algebraic arguments. We also give an explicit example which is related to multiple special $L$-values. Many parts of this problem are still wide open and can be done further in several directions.

One of the things that all families mentioned above have in common is that their Mahler measures can be written in terms of hypergeometric series. Therefore, one can deduce some interesting hypergeometric evaluations from Mahler measure formulas easily. For instance, the equality \eqref{E:first} implies that 
\begin{equation*}
\pFq{4}{3}{\frac{3}{2},\frac{3}{2},1,1}{2,2,2}{-16+12\sqrt{2}}=\frac{4+ 3\sqrt{2}}{2}\left(\log(8+6\sqrt{2})-(L'(f_{64},0)+ L'(f_{32},0))\right).
\end{equation*}

\noindent\textbf{Acknowledgements}\\
The author would like to express his gratitude to Matthew Papanikolas for many helpful discussions and his continuous encouragement during the period of this work. The author would also like to thank Mathew Rogers for useful comments on the preliminary version of this paper and he is indebted to the referee for detailed remarks which help improve many parts of the manuscript.

\section{Two-variable Mahler measures} \label{Sec:two}
As mentioned earlier, we will study Mahler measures of the two families with the complex parameter $t$, namely 
\begin{align*}
m_2(t)&:=2m(P_{t^{1/2}})=2m(x+x^{-1}+y+y^{-1}+t^{1/2}),\\
m_3(t)&:=3m(Q_{t^{1/3}})=3m(x^3+y^3+1-t^{1/3}xy).
\end{align*}
It is known that for most values of $t$ the Mahler measures $m_2(t)$ and $m_3(t)$ can be expressed in terms of hypergeometric series. Indeed, we have the following result: (See, for instance, \cite[Thm.~3.1]{RogersHyper}.)
\begin{proposition}\label{T:Hyper2}
Let $m_2(t)$ and $m_3(t)$ be as defined above.
\begin{enumerate}
\item[(i)]
If $t\neq 0$, then 
$\displaystyle m_2(t)=\Re\left(\log(t)-\frac{4}{t}\pFq{4}{3}{\frac{3}{2},\frac{3}{2},1,1}{2,2,2}{\frac{16}{t}}\right).$

\item[(ii)]
If $|t|\geq 27$, then
$\displaystyle m_3(t)=\Re\left(\log(t)-\frac{6}{t}\pFq{4}{3}{\frac{4}{3},\frac{5}{3},1,1}{2,2,2}{\frac{27}{t}}\right).$
\end{enumerate}
\end{proposition}

Furthermore, Kurokawa and Ochiai \cite{KO} and Lal\'{i}n and Rogers \cite{LR} showed that $m_2(t)$ satisfies some functional equations, which enable us to prove and to conjecture new Mahler measure formulas for some $t\notin \mathbb{Z}.$ Throughout this section, $f_N$ denotes a normalized newform of weight $2$ and level $N$ with rational Fourier coefficients.
\begin{theorem}\label{P:Main2}
The following identities are true:
\begin{align}
m_2(8+6\sqrt{2})&=L'(f_{64},0)+L'(f_{32},0),\label{E:Main21}\\
m_2(8-6\sqrt{2})&=L'(f_{64},0)-L'(f_{32},0),\label{E:Main22}
\end{align}
where $\displaystyle f_{64}(\tau)=\frac{\eta^8(8\tau)}{\eta^2(4\tau) \eta^2(16\tau)}\in S_2(\Gamma_0(64))$ and $f_{32}(\tau)=\eta^2(4\tau) \eta^2(8\tau)\in S_2(\Gamma_0(32)).$ (As usual, let $\eta$ denote the Dedekind eta function,  $$\eta(\tau)=q^{\frac{1}{24}}\prod_{n=1}^{\infty}(1-q^n),$$ where $q=e^{2\pi i \tau}$, and let $S_k(\Gamma_0(N))$ denote the space of cusp forms of weight $k$ and level $N$.)
\end{theorem}
\begin{proof}
It was proved in \cite[Thm.~7]{KO} that if $k\in\mathbb{R}\backslash\{0\}$, then 
\begin{equation}\label{E:KO} 2m_2\left(4\left(k+\frac{1}{k}\right)^2\right)=m_2(16k^4)+m_2\left(\frac{16}{k^4}\right).
\end{equation}
Recall from \eqref{E:RV1} and \eqref{E:RV2} that $m_2(32)=2L'(f_{64},0)$ and $m_2(8)=2L'(f_{32},0)$, so we can deduce \eqref{E:Main21} easily by substituting $k=2^{1/4}$ in \eqref{E:KO}.
On the other hand, one sees from \cite[Thm.~2.2]{LR} that the following functional equation holds for any $k$ such that $0<|k|<1:$
\begin{equation}\label{E:LR}
m_2\left(4\left(k+\frac{1}{k}\right)^2\right)+m_2\left(-4\left(k-\frac{1}{k}\right)^2\right)=m_2\left(\frac{16}{k^4}\right).
\end{equation}
In particular, choosing $k=2^{-1/4}$, we obtain 
$$m_2(8+6\sqrt{2})+m_2(8-6\sqrt{2})=m_2(32).$$ 
Now \eqref{E:Main22} follows immediately from the known information above.
\end{proof}

Rodriguez Villegas \cite[Tab.~4]{RV} verified numerically that $\displaystyle m_2(128)\stackrel{?}=\frac{1}{2}L'(f_{448},0)$ and $\displaystyle m_2(2)=\frac{1}{2}L'(f_{56},0),$ where $f_{448}(\tau)=q-2q^5-q^7-3q^9+4q^{11}-2q^{13}-6q^{17}-\cdots$ and $f_{56}(\tau)=q+2q^5-q^7-3q^9-4q^{11}+2q^{13}-6q^{17}+\cdots$. (In fact,  the latter identity was recently proved by Zudilin \cite{Zudilin}.) Therefore, letting $k=2^{3/4}$ in \eqref{E:KO} and $k=2^{-3/4}$ in \eqref{E:LR} results in a couple of conjectured formulas similar to \eqref{E:Main21} and \eqref{E:Main22}.

\begin{conjecture}
The following identities are true: 
\begin{align*}
m_2\left(8+9\sqrt{2}\right)&\stackrel{?}=\frac{1}{4}\left(L'(f_{448},0)+L'(f_{56},0)\right),\\
m_2\left(8-9\sqrt{2}\right)&\stackrel{?}=\frac{1}{4}\left(L'(f_{448},0)-L'(f_{56},0)\right).
\end{align*}
\end{conjecture}

We also found via numerical computations the following conjectured formulas:
\begin{align*}
m_2\left(\frac{49+9\sqrt{17}}{2}\right)&\stackrel{?}=\frac{1}{2}\left(L'(f_{289},0)+8L'(f_{17},0)\right),\\
m_2\left(\frac{49-9\sqrt{17}}{2}\right)&\stackrel{?}=\frac{1}{2}\left(L'(f_{289},0)-8L'(f_{17},0)\right),
\end{align*}
where $f_{289}(\tau)=q-q^2-q^4+2q^5-4q^7+3q^8-3q^9-\cdots$ and $f_{17}(\tau)=q-q^2-q^4-2q^5+4q^7+3q^8-3q^9+\cdots$. Observe that we can again employ the identity \eqref{E:KO} for $k=(1+\sqrt{17})/4$ to deduce 
\begin{equation*}
2m_2(17)=m_2\left(\frac{49+9\sqrt{17}}{2}\right)+m_2\left(\frac{49-9\sqrt{17}}{2}\right)\stackrel{?}=L'(f_{289},0),
\end{equation*}
which is equivalent to a conjectured formula in \cite[Tab.~4]{RV}. A weaker form of these formulas, namely 
\begin{equation*}
m_2\left(\frac{49+9\sqrt{17}}{2}\right)-m_2(17)\stackrel{?}=4L'(f_{17},0),
\end{equation*}
was also briefly discussed in \cite[\S 4]{RY}.

To study the Mahler measure $m_3(t)$, we use the following crucial result, which basically states that $m_3(t)$ can be written in terms of 
Eisenstein-Kronecker series when $t$ is parameterized properly.

\begin{proposition}[{Rodriguez Villegas \cite[\S IV]{RV}}]\label{T:RV}
Let $\displaystyle t_3(\tau)=27+\left(\frac{\eta(\tau)}{\eta(3\tau)}\right)^{12}$, and let $\mathcal{F}$ be the fundamental domain for $\Gamma_0(3)$ with vertices $i\infty,0,(1+i/\sqrt{3})/2,$ and $(-1+i/\sqrt{3})/2.$ If $\tau\in \mathcal{F},$ then $$m_3(t_3(\tau))=\frac{81\sqrt{3}\Im(\tau)}{4\pi^2}\sideset{}{'}\sum_{m,n\in\mathbb{Z}}\frac{\chi_{-3}(m)(m+3n\Re(\tau))}{[(m+3n\tau)(m+3n\bar{\tau})]^2},$$ where $\displaystyle\sideset{}{'}\sum_{m,n}$ means that $(m,n)=(0,0)$ is excluded from the summation.
\end{proposition}

The remaining part of this section will be devoted to proving the following result:

\begin{theorem} \label{T:Main3}
If $t=6-6\sqrt[3]{2}+18\sqrt[3]{4},$ then 
$$m_3(t)=\frac{3}{2}\left(L'(f_{108},0)+L'(f_{36},0)-3L'(f_{27},0)\right),$$
where $f_{36}(\tau)=\eta^4(6\tau)\in S_2(\Gamma_0(36))$, $f_{27}(\tau)=\eta^2(3\tau) \eta^2(9\tau)\in S_2(\Gamma_0(27)),$ and $f_{108}(\tau)=q+5q^7-7q^{13}-q^{19}-5q^{25}-4q^{31}-q^{37}+\cdots$, the unique normalized newform in $S_2(\Gamma_0(108))$.
\end{theorem}

Applying Theorem~\ref{T:Hyper2}, Proposition~\ref{P:Main2}, and Theorem~\ref{T:Main3}, one obtains the following hypergeometric evaluation formulas immediately:
\begin{corollary}\label{C:hyperL}
The following identities hold:
\begin{align*}
\pFq{4}{3}{\frac{3}{2},\frac{3}{2},1,1}{2,2,2}{-16+12\sqrt{2}}&=\frac{4+ 3\sqrt{2}}{2}\left(\log(8+6\sqrt{2})-(L'(f_{64},0)+ L'(f_{32},0))\right),\\
\pFq{4}{3}{\frac{4}{3},\frac{5}{3},1,1}{2,2,2}{\frac{63+171\sqrt[3]{2}-18\sqrt[3]{4}}{250}}&=\left(1-\sqrt[3]{2}+3\sqrt[3]{4}\right)\biggl(\log(6-6\sqrt[3]{2}+18\sqrt[3]{4})\\&\qquad -\frac{3}{2}\left(L'(f_{108},0)+L'(f_{36},0)-3L'(f_{27},0)\right)\biggr).
\end{align*}
\end{corollary}

To establish Theorem~\ref{T:Main3}, we require some identities for $L$-values of the involved cusp forms, which will be verified in the following lemmas.

\begin{lemma}\label{L:f36}
Let $f_{36}(\tau)$ be as defined in Theorem~\ref{T:Main3}. Then the following equality holds:
\begin{equation*}
L(f_{36},2)=\frac{1}{2}\sideset{}{'}\sum_{m,n\in\mathbb{Z}}\frac{m\chi_{-3}(m)}{(m^2+3n^2)^2}.
\end{equation*}
\end{lemma}
\begin{proof}
First, note that for any $\tau$ in the upper half plane $\eta(\tau)$ satisfies the functional equation
$$\eta\left(-\frac{1}{\tau}\right)=\sqrt{-i\tau}\eta(\tau).$$ Hence it is easily seen that 
$$\frac{\eta\left(\frac{\sqrt{-3}}{3}\right)}{\eta\left(\sqrt{-3}\right)}=3^\frac{1}{4},$$
which implies that $\displaystyle t_3\left(\frac{\sqrt{-3}}{3}\right)=54.$ Thus we have from Theorem~\ref{T:RV} that 
\begin{equation*}
m_3(54)=\frac{81}{4\pi^2}\sideset{}{'}\sum_{m,n\in\mathbb{Z}}\frac{m\chi_{-3}(m)}{(m^2+3n^2)^2}.
\end{equation*}
On the other hand, Rogers \cite[Thm.~2.1, Thm.~5.2]{RogersHyper} proved that 
\begin{equation*}
m_3(54)=\frac{81}{2\pi^2}L(f_{36},2),
\end{equation*}
whence the lemma follows. 
\end{proof}

\begin{lemma}\label{L:f108}
Let $f_{108}(\tau)$ be the unique normalized newform with rational coefficients in $S_2(\Gamma_0(108))$, and let $\mathcal{A}=\{(m,n)\in \mathbb{Z}^2 \mid (m,n) \equiv (-1,-2),(2,1),(1,0),(-2,3) \mod 6\}.$ Then 
\begin{equation*}
L(f_{108},2)=\sum_{m,n\in\mathcal{A}}\frac{m+3n}{(m^2+3n^2)^2}.
\end{equation*}
\end{lemma}
\begin{proof}
By taking the Mellin transform of the newform, it suffices to prove that 
\begin{equation}\label{E:f108}
f_{108}(\tau)=\sum_{m,n\in\mathcal{A}}(m+3n)q^{m^2+3n^2}.
\end{equation}
Let $K=\mathbb{Q}(\sqrt{-3}), \mathcal{O}_{K}=\mathbb{Z}\left[\frac{1+\sqrt{-3}}{2}\right], \Lambda=(3+3\sqrt{-3})\subset \mathcal{O}_K$, and $I(\Lambda)$ = the group of fractional ideals of $\mathcal{O}_K$ coprime to $\Lambda$. Since $\Lambda$ can be factorized as 
\begin{equation*}
\Lambda = \left(\frac{1+\sqrt{-3}}{2}\right)(\sqrt{-3})^2(2), 
\end{equation*}
any integral ideal $\mathfrak{a}$ is coprime to $\Lambda$ if and only if $(\sqrt{-3})\nmid \mathfrak{a}$ and $(2)\nmid \mathfrak{a}$. As a consequence, every integral ideal coprime to $\Lambda$ is uniquely represented by $(m+n\sqrt{-3})$, where $m,n\in\mathbb{Z}, m>0, 3\nmid m,$ and $m \not\equiv n \pmod 2.$ Let $P(\Lambda)$ denote the monoid of integral ideals coprime to $\Lambda$.

Define $\varphi:P(\Lambda)\rightarrow \mathbb{C}^{\times}$ by 
\begin{align*}
\varphi((m+n\sqrt{-3}))= 
	\begin{cases} 
	\displaystyle
	\frac{-\chi_{-3}(m) m+\chi_{-3}(n)(3n)-(\chi_{-3}(n)m+\chi_{-3}(m)n)\sqrt{-3}}{2} & \mbox{if } 3\nmid n, \\ 			
	\chi_{-3}(m)(m+n\sqrt{-3}) & \mbox{if } 3|n. 
	\end{cases}
\end{align*}
Then it is not difficult to check that $\varphi$ is multiplicative, and for each $(m+n\sqrt{-3})\in P(\Lambda)$ with $m+n\sqrt{-3}\equiv 1 \pmod \Lambda,$
$$\varphi((m+n\sqrt{-3}))=m+n\sqrt{-3}.$$
Hence we can extend $\varphi$ multiplicatively to define a Hecke Gr\"{o}ssencharacter of weight $2$ and conductor $\Lambda$ on $I(\Lambda).$ Now if we let $$\Psi(\tau):=\sum_{\mathfrak{a}\in P(\Lambda)}\varphi(\mathfrak{a})q^{N(\mathfrak{a})},$$
then one sees from \cite[Thm.~1.31]{Ono} that $\Psi(\tau)$ is a newform in $S_2(\Gamma_0(108)).$ Observe that 
\begin{align*}
\varphi((m+n\sqrt{-3}))+\varphi((m-n\sqrt{-3}))=
	\begin{cases} 
	-\chi_{-3}(m) m+\chi_{-3}(n)(3n) & \mbox{if } 3\nmid n, \\ 			
	2\chi_{-3}(m)m & \mbox{if } 3|n, 
	\end{cases}
\end{align*}
so we have 
\begin{equation*}
\Psi(\tau)=\sum_{\substack{m,n\in\mathbb{N} \\ 3\nmid m, 3\nmid n \\ m\not\equiv n \pmod 2}}(-\chi_{-3}(m) m+\chi_{-3}(n)(3n))q^{m^2+3n^2}+\sum_{\substack{m\in\mathbb{N}, n\in\mathbb{Z}\\ 3\nmid m, 3|n \\ m\not\equiv n \pmod 2}}\chi_{-3}(m)m q^{m^2+3n^2}. 
\end{equation*}
Working modulo $6$, one can show that 
\begin{equation*}
\sum_{\substack{m,n\in\mathbb{N} \\ 3\nmid m, 3\nmid n \\ m\not\equiv n \pmod 2}}(-\chi_{-3}(m) m+\chi_{-3}(n)(3n))q^{m^2+3n^2} = \sum_{\substack{m,n\in\mathbb{Z} \\ (m,n)\equiv (-1,2),(2,1) \\ \pmod 6}}( m+3n)q^{m^2+3n^2}, 
\end{equation*}
and 
\begin{align*}
\sum_{\substack{m\in\mathbb{N}, n\in\mathbb{Z}\\ 3\nmid m, 3|n \\ m\not\equiv n \pmod 2}}\chi_{-3}(m)m q^{m^2+3n^2}&= \sum_{\substack{m,n\in\mathbb{Z}\\(m,n)\equiv (1,0),(-2,3) \\ \pmod 6}}mq^{m^2+3n^2}\\
&=\sum_{\substack{m,n\in\mathbb{Z}\\(m,n)\equiv (1,0),(-2,3) \\ \pmod 6}}(m+3n)q^{m^2+3n^2}.
\end{align*}
Consequently, the coefficients of $\Psi(\tau)$ are rational, which implies that $\Psi(\tau)=f_{108}(\tau)$, and \eqref{E:f108} holds. (One can check using, for example, \texttt{Sage} or \texttt{Magma} that there is only one normalized newform in $S_2(\Gamma_0(108))$.)
\end{proof}

\begin{lemma}\label{L:f27}
Let $f_{27}(\tau)$ be as defined in Theorem~\ref{T:Main3}, and let $\mathcal{B}=\{(m,n)\in \mathbb{Z}^2 \mid (m,n) \equiv (1,0),(-2,3),(1,-1),(-2,2),(2,-1),(-1,2) \mod 6\}.$ Then 
\begin{equation*}
L(f_{27},2)=\sideset{}{'}\sum_{m,n\in\mathcal{B}}\frac{m+3n}{(m^2+3n^2)^2}.
\end{equation*}
\end{lemma}
\begin{proof}
As before, we will establish a $q$-expansion for $f_{27}(\tau)$ first; i.e., we aim at proving that 
\begin{equation*}
f_{27}(\tau)=\sum_{m,n\in\mathcal{B}}(m+3n)q^{m^2+3n^2}.
\end{equation*}
Recall from \cite[\S 6]{RogersHyper} that the following identity is true:
\begin{equation}\label{E:Rf27}
f_{27}(\tau)=\sum_{\substack{m,n\in\mathbb{Z}\\ (m,n)\equiv (1,1),(-2,-2)\\ \pmod 6 }}\left(\frac{m+3n}{4}\right)q^{ \frac{m^2+3n^2}{4}}.
\end{equation}
Therefore, it is sufficient to prove the following claims, each of which involves only simple manipulation. (Unless otherwise stated, each ordered pair $(a,b)$ listed beneath the sigma sign indicates all $(m,n)\in\mathbb{Z}^2$ such that $m\equiv a$ and $n\equiv b \pmod 6.$) 

\begin{claim}\label{Cl:1}
\begin{equation*}
\sum_{(1,1)}\left(\frac{m+3n}{4}\right)q^{ \frac{m^2+3n^2}{4}} = \sum_{(1,0),(-2,3)}(m+3n)q^{m^2+3n^2}+\sum_{(2,-1),(-1,2)}\left(\frac{m+3n}{2}\right)q^{m^2+3n^2}.
\end{equation*}
\end{claim}
\begin{claim}\label{Cl:2}
\begin{equation*}
\sum_{(-2,-2)}\left(\frac{m+3n}{4}\right)q^{ \frac{m^2+3n^2}{4}} = \sum_{(1,-1),(-2,2)}(m+3n)q^{m^2+3n^2}+\sum_{(2,-1),(-1,2)}\left(\frac{m+3n}{2}\right)q^{m^2+3n^2}.
\end{equation*}
\end{claim}
\begin{proof}[Proof of Claim~\ref{Cl:1}]
It is clear that 
\begin{align*}
\sum_{(1,0),(-2,3)}(m+3n)q^{m^2+3n^2}&=\sum_{(1,0),(-2,3)}mq^{m^2+3n^2}\\&=\sum_{(1,0),(-2,3)}\left(\frac{(m+3n)+3(m-n)}{4}\right)q^{\frac{(m+3n)^2+3(m-n)^2}{4}}, \text{ and }
\end{align*}
\begin{equation*}
\sum_{(2,-1),(-1,2)}\left(\frac{m+3n}{2}\right)q^{m^2+3n^2}=\sum_{(2,-1),(-1,2)}\left(\frac{(3n-m)+3(m+n)}{4}\right)q^{\frac{(3n-m)^2+3(m+n)^2}{4}}.
\end{equation*}
Also, it can be verified in a straightforward manner that 
\begin{multline*}
\{(m,n) \mid m\equiv n \equiv 1 \pmod 6\}=\{(k+3l,k-l)\mid (k,l)\equiv (1,0),(-2,3) \pmod 6\}\\ \sqcup \{(3l-k,k+l)\mid (k,l)\equiv (2,-1),(-1,2) \pmod 6\},
\end{multline*}
where $\sqcup$ denotes disjoint union, so we obtain Claim~\ref{Cl:1}.
\renewcommand{\qedsymbol}{}
\end{proof}

\begin{proof}[Proof of Claim~\ref{Cl:2}]
Let us make some observation first that, by symmetry, 
\begin{equation*}
\sum_{(1,-1),(-2,2)}(3m+3n)q^{m^2+3n^2}=0,
\end{equation*}
so we have that 
\begin{equation*}
\sum_{(1,-1),(-2,2)}(-2m)q^{m^2+3n^2}=\sum_{(1,-1),(-2,2)}(m+3n)q^{m^2+3n^2}.
\end{equation*}
It follows that
\begin{equation}\label{E:ad1}
\begin{aligned}
\sum_{(-1,-1),(2,2)}(m+3n)q^{m^2+3n^2}&=\sum_{(1,-1),(-2,2)}(-m+3n)q^{m^2+3n^2}\\
&=\sum_{(1,-1),(-2,2)}(m+3n)q^{m^2+3n^2}+\sum_{(1,-1),(-2,2)}(-2m)q^{m^2+3n^2}\\
&=2\sum_{(1,-1),(-2,2)}(m+3n)q^{m^2+3n^2}.
\end{aligned}
\end{equation}
Therefore, 
\begin{align*}
\sum_{(-2,-2)}\left(\frac{m+3n}{4}\right)q^{ \frac{m^2+3n^2}{4}}&=\sum_{(-1,-1) \pmod 3}\left(\frac{m+3n}{2}\right)q^{m^2+3n^2}\\
&=\sum_{\substack{(-1,-1),(2,2)\\ (2,-1),(-1,2)}}\left(\frac{m+3n}{2}\right)q^{m^2+3n^2}\\
&=\sum_{(1,-1),(-2,2)}(m+3n)q^{m^2+3n^2}+\sum_{(2,-1),(-1,2)}\left(\frac{m+3n}{2}\right)q^{m^2+3n^2},
\end{align*}
where the last equality comes from \eqref{E:ad1}.
\renewcommand{\qedsymbol}{}
\end{proof}
\end{proof}

\begin{lemma}\label{L:f108f27}
The following equality is true:
\begin{equation*}
L(f_{108},2)-\frac{3}{4}L(f_{27},2)=\frac{3}{2}\sum_{\substack{m,n\in\mathbb{Z}\\ 3\nmid n}}\frac{m\chi_{-3}(m)}{(3m^2+n^2)^2}.
\end{equation*}
\end{lemma}
\begin{proof}
Taking the Mellin transform of $f_{27}(\tau)$ in \eqref{E:Rf27} yields 
\begin{equation}\label{E:Lf27}
L(f_{27},2)=4\sum_{(1,1),(-2,-2)}\frac{m+3n}{(3m^2+n^2)^2}.
\end{equation}
Since $\chi_{-3}(n)=j$ iff $n\equiv j \pmod 3$, where $j\in\{-1,0,1\}$,
we have that 
\begin{align*}
\sum_{\substack{m,n\in\mathbb{Z}\\ 3\nmid n}}\frac{m\chi_{-3}(m)}{(3m^2+n^2)^2}&=\sum_{\substack{m,n\in\mathbb{Z}\\ 3\nmid m}}\frac{n\chi_{-3}(n)}{(m^2+3n^2)^2}\\
&=\sum_{\substack{n\equiv -1 \pmod 3\\ 3\nmid m}}\frac{-2n}{(m^2+3n^2)^2}.
\end{align*}
Also, it is obvious that the symmetry of the summation yields 
\begin{equation*}
\sum_{\substack{n\equiv -1 \pmod 3\\ 3\nmid m}}\frac{m}{(m^2+3n^2)^2}=0.
\end{equation*}
Hence, using Lemma~\ref{L:f108}, one sees that 
\begin{align*}
L(f_{108},2)-\frac{3}{2}\sum_{\substack{m,n\in\mathbb{Z}\\ 3\nmid n}}\frac{m\chi_{-3}(m)}{(3m^2+n^2)^2} &= \sum_{\substack{(-1,-2),(2,1)\\ (1,0),(-2,3)}}\frac{m+3n}{(m^2+3n^2)^2}+\sum_{\substack{n\equiv -1 \pmod 3\\ 3\nmid m}}\frac{3n}{(m^2+3n^2)^2}\\
&=\sum_{\substack{(-1,-2),(2,1)\\ (1,0),(-2,3)}}\frac{m+3n}{(m^2+3n^2)^2}+\sum_{\substack{n\equiv -1 \pmod 3\\ 3\nmid m}}\frac{m+3n}{(m^2+3n^2)^2}\\
&=\sum_{\substack{(-1,-2),(2,1)\\ (1,0),(-2,3)}}\frac{m+3n}{(m^2+3n^2)^2}+\sum_{\substack{(-2,2),(-2,-1)\\(-1,2),(-1,-1)\\(1,2),(1,-1)\\(2,2),(2,-1)}}\frac{m+3n}{(m^2+3n^2)^2}\\
&= \sum_{\substack{(1,0),(-2,3)\\(1,-1),(-2,2)\\(2,-1),(-1,2)}}\frac{m+3n}{(m^2+3n^2)^2}-\sum_{(1,1),(-2,-2)}\frac{m+3n}{(m^2+3n^2)^2}\\
&= L(f_{27},2)-\frac{1}{4}L(f_{27},2),
\end{align*}
where we have applied Lemma~\ref{L:f27} and \eqref{E:Lf27} in the last equality.
\end{proof}
Putting the previous lemmas together, we are now ready to complete a proof of Theorem~\ref{T:Main3}.

\begin{proof}[Proof of Theorem~\ref{T:Main3}]
Let $\tau_0=\sqrt{-3}/9$. Then $t_3(\tau_0)=6-6\sqrt[3]{2}+18\sqrt[3]{4}.$ This can be verified by considering numerical approximation of $t_3(\tau_0)$ and using the following identities:
\begin{align*}
j(\tau)&=j(-1/\tau),\qquad \mathfrak{f}^3(\sqrt{-27})=2(1+\sqrt[3]{2}+\sqrt[3]{4}),\\
j(\tau)&=\frac{(\mathfrak{f}^{24}(\tau)-16)^3}{\mathfrak{f}^{24}(\tau)}=\frac{t_3(\tau)(t_3(\tau)+216)^3}{(t_3(\tau)-27)^3},
\end{align*}
where $j(\tau)$ is the $j$-invariant, and $\mathfrak{f}(\tau)$ is a Weber modular function defined by $$\mathfrak{f}(\tau)=e^{-\frac{\pi i}{24}}\frac{\eta\left(\frac{\tau+1}{2}\right)}{\eta(\tau)}.$$ (For references to these identities, see \cite[\S1]{ChenYui}, \cite[Tab.~VI]{Weber},  and \cite[\S 1]{YuiZagier}.)
Then we see from Proposition~\ref{T:RV} that  
\begin{align*}
m_3(t_3(\tau_0))&=\frac{27}{4\pi^2}\sideset{}{'}\sum_{m,n\in\mathbb{Z}}\frac{m\chi_{-3}(m)}{(m^2+\frac{n^2}{3})^2}\\
&=\frac{3}{2}\left(\frac{81}{2\pi^2}\sideset{}{'}\sum_{m,n\in\mathbb{Z}}\frac{m\chi_{-3}(m)}{(3m^2+n^2)^2}\right)\\
&=\frac{3}{2}\left(\frac{81}{2\pi^2}\sideset{}{'}\sum_{\substack{m,n\in\mathbb{Z}\\3|n}}\frac{m\chi_{-3}(m)}{(3m^2+n^2)^2}+\frac{81}{2\pi^2}\sum_{\substack{m,n\in\mathbb{Z}\\3\nmid n}}\frac{m\chi_{-3}(m)}{(3m^2+n^2)^2}\right)\\
&=\frac{3}{2}\left(\frac{9}{2\pi^2}\sideset{}{'}\sum_{m,n\in\mathbb{Z}}\frac{m\chi_{-3}(m)}{(m^2+3n^2)^2}+\frac{81}{2\pi^2}\sum_{\substack{m,n\in\mathbb{Z}\\3\nmid n}}\frac{m\chi_{-3}(m)}{(3m^2+n^2)^2}\right).
\end{align*}
Now we can deduce using Lemma~\ref{L:f36} and Lemma~\ref{L:f108f27} that 
\begin{equation}\label{E:Lf2}
m_3(t_3(\tau_0))=\frac{3}{2}\left(\frac{27}{\pi^2}L(f_{108},2)+\frac{9}{\pi^2}L(f_{36},2)-\frac{81}{4\pi^2}L(f_{27},2)\right).
\end{equation}
Finally, the formula stated in the theorem is merely a simple consequence of \eqref{E:Lf2} and the functional equation
\begin{equation*}\label{E:FEnewform}
\left(\frac{\sqrt{N}}{2\pi}\right)^s\Gamma(s)L(f,s)=\epsilon\left(\frac{\sqrt{N}}{2\pi}\right)^{2-s}\Gamma(2-s)L(f,2-s),
\end{equation*}
where $f$ is any newform of weight $2$ and level $N$ with real Fourier coefficients, and $\epsilon\in\{-1,1\}$, depending on $f$. (If $f\in\{f_{27},f_{36},f_{108}\}$, then $\epsilon=1$.)
\end{proof}
In addition to the formula stated in Theorem~\ref{T:Main3}, we discovered some other conjectured formulas of similar type using numerical values of the hypergeometric representation of $m_3(t)$ given in Proposition~\ref{T:Hyper2}:
\begin{align*}
m_3\left(17766+14094\sqrt[3]{2}+11178\sqrt[3]{4}\right)&\stackrel{?}=\frac{3}{2}(L'(f_{108},0)+3L'(f_{36},0)+3L'(f_{27},0)),\\
m_3(\alpha\pm\beta i)&\stackrel{?}=\frac{3}{2}(L'(f_{108},0)+3L'(f_{36},0)-6L'(f_{27},0)),\\ 
m_3\left(\frac{(7+\sqrt{5})^3}{4}\right)&\stackrel{?}=\frac{1}{8}\left(9L'(f_{100},0)+38L'(f_{20},0)\right),\\
m_3\left(\frac{(7-\sqrt{5})^3}{4}\right)&\stackrel{?}=\frac{1}{4}\left(9L'(f_{100},0)-38L'(f_{20},0)\right),
\end{align*} 
where $\alpha=17766-7047\sqrt[3]{2}-5589\sqrt[3]{4}$,  $\beta=27\sqrt{3}(261\sqrt[3]{2}-207\sqrt[3]{4})$, $f_{100}(\tau)=q+2q^3-2q^7+q^9-2q^{13}+6q^{17}-4q^{19}-\cdots$, and  $f_{20}(\tau)=\eta^2(2\tau)\eta^2(10\tau)$.

It is worth mentioning that the last two Mahler measures above also appear in \cite[Thm.~6]{Guillera} and \cite[\S 4]{RZ0}. More precisely, it was shown that 
\begin{align}
19m_3(32)\label{E:WZ}&=16m_3\left(\frac{(7+\sqrt{5})^3}{4}\right)-8m_3\left(\frac{(7-\sqrt{5})^3}{4}\right),\\
m_3(32)&=8L'(f_{20},0).
\end{align}
Many of the identities like \eqref{E:WZ} can be proved using the elliptic dilogarithm evaluated at some torsion points on the corresponding elliptic curve. However, to our knowledge, no rigorous proof of the conjectured formulas for the individual terms on the right seems to appear in the literature.

\section{Three-variable Mahler measures} \label{Sec:three}
From here on, we denote 
\begin{align*}
A_s &:=(x+x^{-1})(y+y^{-1})(z+z^{-1})+s^{1/2}, &n_2(s)&:=2m(A_s),\\
B_s &:=(x+x^{-1})^2(y+y^{-1})^2(1+z)^3z^{-2}-s, &n_3(s)&:=m(B_s),\\
C_s &:=x^4+y^4+z^4+1+s^{1/4}xyz, &n_4(s)&:=4m(C_s),\\
s_2(q(\tau))&:=-\frac{\Delta\left(\tau+\frac{1}{2}\right)}{\Delta(2\tau+1)},\\
s_3(q(\tau))&:=\left(27\left(\frac{\eta(3\tau)}{\eta(\tau)}\right)^6+\left(\frac{\eta(\tau)}{\eta(3\tau)}\right)^6\right)^2,\\
s_4(q(\tau))&:=\frac{\Delta(2\tau)}{\Delta(\tau)}\left(16\left(\frac{\eta(\tau)\eta(4\tau)^2}{\eta(2\tau)^3}\right)^4+\left(\frac{\eta(2\tau)^3}{\eta(\tau)\eta(4\tau)^2}\right)^4\right)^4,
\end{align*}
where $\Delta(\tau)=\eta^{24}(\tau)$ and $q(\tau)=e^{2\pi i \tau}.$ By abuse of notation, we will sometimes write $s_j(\tau)$ instead of $s_j(q(\tau))$, while they actually represent the same function.

The main result we will show in this section is stated as follows:
\begin{theorem}\label{P:threevar}
The following identities are true:
\begin{align*}
n_4(26856+15300\sqrt{3})&=\frac{5}{12}\left(20L'(g_{12},0)+4L'(g_{48},0)+11L'(\chi_{-3},-1)+8L'(\chi_{-4},-1)\right),\\
n_4(26856-15300\sqrt{3})&=\frac{5}{6}\left(-20L'(g_{12},0)+4L'(g_{48},0)-11L'(\chi_{-3},-1)+8L'(\chi_{-4},-1)\right),
\end{align*}
where $g_{12}(\tau)=\eta^3(2\tau)\eta^3(6\tau)\in S_3(\Gamma_0(12),\chi_{-3})$, and $g_{48}(\tau)$ is the quadratic twist of $g_{12}$ by $\chi_{-4}$ and belongs to $S_3(\Gamma_0(48),\chi_{-3}).$
\end{theorem}
\begin{proof}
By a result in \cite[Prop.~2.1]{Samart}, we have that $n_4(s)$ can be expressed as Eisenstein-Kronecker series when $s$ is parameterized by $s_4(\tau)$, namely
\begin{multline}\label{E:EKn4}
n_4(s_4(\tau))=\frac{10\Im(\tau)}{\pi^3}\sideset{}{'}\sum_{m,n\in \mathbb{Z}}\biggl(-\left(\frac{4n^2}{(m^2|\tau|^2+n^2)^3}-\frac{1}{(m^2|\tau|^2+n^2)^2}\right)\\
+4\left(\frac{4n^2}{(4m^2|\tau|^2+n^2)^3}-\frac{1}{(4m^2|\tau|^2+n^2)^2}\right)\biggr)
\end{multline}
for every $\tau\in\mathbb{C}$ such that $\tau$ is purely imaginary and $\Im(\tau)\geq 1/ \sqrt{2}.$ It is clear that $s_4(\tau)$ can be rewritten in the form 
\begin{equation*}
s_4(\tau)=\frac{1}{\mathfrak{f}_1^8(2\tau)}\left(\frac{16}{\mathfrak{f}_1^8(4\tau)}+\frac{\mathfrak{f}_1^8(4\tau)}{\mathfrak{f}_1^{8}(2\tau)}\right)^4,
\end{equation*}
where $\displaystyle\mathfrak{f}_1(\tau):=\frac{\eta\left(\frac{\tau}{2}\right)}{\eta(\tau)}$, also known as a Weber modular function. We obtain from \cite[Tab.~VI]{Weber} that 
\begin{equation*}
\mathfrak{f}_1^4(\sqrt{-12})=2^{\frac{7}{6}}(1+\sqrt{3}), \qquad \mathfrak{f}_1^8(\sqrt{-48})=2^{\frac{19}{6}}(1+\sqrt{3})(\sqrt{2}+\sqrt{3})^2(1+\sqrt{2})^2.
\end{equation*}
Therefore, after simplifying, we have $s_4\left(\sqrt{-3}\right)=26856+15300\sqrt{3},$ and substituiting $\tau=\sqrt{-3}$ in \eqref{E:EKn4} yields
\begin{equation}\label{E:ad2}
\begin{aligned}
n_4(26856+15300\sqrt{3})&=\frac{10\sqrt{3}}{\pi^3}\sideset{}{'}\sum_{m,n\in\mathbb{Z}}\biggl(-\left(\frac{4n^2}{(3m^2+n^2)^3}-\frac{1}{(3m^2+n^2)^2}\right)\\&\qquad\qquad+4\left(\frac{4n^2}{(12m^2+n^2)^3}-\frac{1}{(12m^2+n^2)^2}\right)\biggr)\\
&=\frac{10\sqrt{3}}{\pi^3}\sideset{}{'}\sum_{m,n\in\mathbb{Z}}\biggl(\frac{2(3n^2-m^2)}{(m^2+3n^2)^3}+\frac{8(m^2-12n^2)}{(m^2+12n^2)^3}\\&\qquad\qquad+\frac{4}{(m^2+12n^2)^2}-\frac{1}{(m^2+3n^2)^2}\biggr).
\end{aligned}
\end{equation}

It was proved in \cite[Cor.~4.4]{BertinMain} that the following identity holds:
\begin{equation}\label{E:ad3}
\frac{9}{8}\sideset{}{'}\sum_{m,n\in\mathbb{Z}}\frac{m^2-3n^2}{(m^2+3n^2)^3} = \sideset{}{'}\sum_{m,n\in\mathbb{Z}}\left(\frac{m^2-12n^2}{(m^2+12n^2)^3}+\frac{4n^2-3m^2}{(3m^2+4n^2)^3}\right).
\end{equation}
Equivalently, one has that
\begin{equation}\label{E:modular}
\begin{aligned}
\sideset{}{'}\sum_{m,n\in\mathbb{Z}}\left(\frac{2(3n^2-m^2)}{(m^2+3n^2)^3}+\frac{8(m^2-12n^2)}{(m^2+12n^2)^3}\right)&= \frac{5}{2}\sideset{}{'}\sum_{m,n\in\mathbb{Z}}\frac{m^2-3n^2}{(m^2+3n^2)^3}\\&\qquad\qquad+4\sideset{}{'}\sum_{m,n\in\mathbb{Z}}\left(\frac{m^2-12n^2}{(m^2+12n^2)^3}+\frac{3m^2-4n^2}{(3m^2+4n^2)^3}\right)\\
&=5L(g_{12},3)+8L(g_{48},3),
\end{aligned}
\end{equation}
where the last equality is a direct consequence of Lemma~2.7 and Lemma~2.12 in \cite{Samart}. 

Recall from Glasser and Zucker's results on lattice sums \cite[Tab.~VI]{Lattice} that
\begin{equation}\label{E:Dirichlet}
\begin{aligned}
\sideset{}{'}\sum_{m,n\in\mathbb{Z}}\frac{1}{(m^2+3n^2)^2}&=\frac{9}{4}\zeta(2)L(\chi_{-3},2)=\frac{3\pi^2}{8}L(\chi_{-3},2),\\
\sideset{}{'}\sum_{m,n\in\mathbb{Z}}\frac{1}{(m^2+12n^2)^2}&=\frac{69}{64}\zeta(2)L(\chi_{-3},2)+L(\chi_{12},2)L(\chi_{-4},2)\\ 
&=\frac{23\pi^2}{128}L(\chi_{-3},2)+\frac{\pi^2}{6\sqrt{3}}L(\chi_{-4},2).
\end{aligned}
\end{equation}
Then we substitute \eqref{E:modular} and \eqref{E:Dirichlet} in \eqref{E:ad2} to get
\begin{equation*}
n_4(26856+15300\sqrt{3})=\frac{50\sqrt{3}}{\pi^3}L(g_{12},3)+\frac{80\sqrt{3}}{\pi^3}L(g_{48},3)+\frac{55\sqrt{3}}{16\pi}L(\chi_{-3},2)+\frac{20}{3\pi}L(\chi_{-4},2).
\end{equation*}
Finally, the derivative expression follows directly from the functional equations for the involved $L$-functions.

The second formula can be shown in a similar manner by choosing $\tau_0=\sqrt{-3}/2$. Although Weber did not list an explicit value of $\mathfrak{f}_1(\sqrt{-3})$ in his book, one can find it easily using the identity $\mathfrak{f}_1(2\tau)=\mathfrak{f}(\tau)\mathfrak{f}_1(\tau)$ and the fact that $\mathfrak{f}(\sqrt{-3})=2^{\frac{1}{3}}.$ Therefore, we have $s_4(\tau_0)=26856-15300\sqrt{3},$ and 
\begin{align*}
n_4(s_4(\tau_0))&=\frac{20\sqrt{3}}{\pi^3}\sideset{}{'}\sum_{m,n\in\mathbb{Z}}\biggl(\frac{8(3m^2-4n^2)}{(3m^2+4n^2)^3}+\frac{2(m^2-3n^2)}{(m^2+3n^2)^3}+\frac{1}{(m^2+3n^2)^2}-\frac{4}{(3m^2+4n^2)^2}\biggr)\\
&=\frac{20\sqrt{3}}{\pi^3}(-5L(g_{12},3)+8L(g_{48},3)-\frac{11\pi^2}{32}L(\chi_{-3},2)+\frac{2\pi^2}{3\sqrt{3}}L(\chi_{-4},2)),
\end{align*}
where we again use \eqref{E:ad3}, \eqref{E:Dirichlet}, and the identity 
\begin{equation*}
2L(\chi_{12},2)L(\chi_{-4},2)=\sideset{}{'}\sum_{m,n\in\mathbb{Z}}\left(\frac{1}{(m^2+12n^2)^2}-\frac{1}{(3m^2+4n^2)^2}\right)
\end{equation*}
(see \cite[Lem.~2.6]{Samart}).
\end{proof}

\section{Arithmetic of $K3$ surfaces} \label{Sec:four}
For additional details omitted from this section, the reader may consult \cite[\S 2]{BertinMain} and \cite{SchuttII}. Recall that a smooth projective surface $X$ is called a $K3$ surface if $H^1(X,\mathcal{O}_X)=0$ and the canonical bundle of $X$ is trivial. Hence every $K3$ surface admits a holomorphic $2$-form, unique up to scalar multiplication. Also, one has that $H_2(X,\mathbb{Z})$ is a free abelian group of rank $22$ and can be decomposed into $H_2(X,\mathbb{Z})\cong \NS(X)\oplus \T(X),$ where $\NS(X)$, called the \textit{N\'{e}ron-Severi group}, is the group of algebraic equivalence classes of divisors on $X$, and $T(X)$, the \textit{transcendental lattice}, is the orthogonal complement of $\NS(X)$ in $H_2(X,\mathbb{Z}).$ The rank of $\NS(X)$, denoted by $\rho(X)$, is called the \textit{Picard number} of $X$. Over any field of characteristic zero, we have $1\leq \rho(X)\leq 20,$ and $X$ is said to be \textit{singular} if $\rho(X)=20.$  Let $\{\gamma_1,\gamma_2,\ldots,\gamma_{22}\}$ be a basis for $H_2(X,\mathbb{Z})$, and let $\omega$ be a nowhere-vanishing holomorphic $2$-form on $X$. Then the integral $$\int_{\gamma_i}\omega$$
is called a period of $X$, which vanishes if and only if $\gamma_i\in \NS(X).$ We shall denote by $X_s,Y_s,$ and $Z_s$ the projective hypersurfaces corresponding to the one-parameter families $A_s, B_s,$ and $C_s$, respectively. The family $Z_s$ is sometimes called the \textit{Dwork family} and is known to be $K3$ surfaces (see; e.g., \cite{Hartmann}). To see that, for all but finitely many $s$, $X_s$ is a $K3$ surface, it suffices to show that it is birational to an elliptic surface which has a minimal Weierstrass form 
\begin{equation*}
y^2=x^3+A_4(z)x+A_6(z),
\end{equation*}
where $A_4(z),A_6(z)\in \mathbb{Z}[s,z]$ with $\deg (A_i)\leq 2i$ for all $i$ and $\deg (A_i)>i$ for some $i$ \cite[\S 4]{SS}. Indeed, one can manipulate this using \texttt{Maple} and find that 
\begin{align*}
A_4(z)&=-768\left(z^2+1\right)^4+48sz^2\left(z^2+1\right)^2-3s^2z^4,\\
A_6(z)&=8192\left(z^2+1\right)^6-768sz^2\left(z^2+1\right)^4-48s^2z^4\left(z^2+1\right)^2+2s^3z^6.
\end{align*}
Since $A_4(z)$ and $A_6(z)$ satisfy the conditions above, it follows that $X_s$ is generically a family of $K3$ surfaces. Also, using the Weierstrass model above, we have that $X_s$ is defined over $\mathbb{Q}$ if $s\in\mathbb{Q}$.
Letting $s=1/\mu$, we have that a period of $X_{s(\mu)}$ is 
\begin{align*}
u_0(\mu):=&\frac{1}{(2\pi i)^3}\int_{\mathbb{T}^3}\frac{1}{1-\mu^{1/2}\left(x+x^{-1}\right)\left(y+y^{-1}\right)\left(z+z^{-1}\right)}\frac{dx}{x}\frac{dy}{y}\frac{dz}{z}\\
=&\pFq{3}{2}{\frac{1}{2},\frac{1}{2},\frac{1}{2}}{1,1}{64\mu}.
\end{align*}
One can observe from the definition of the Mahler measure that in this case, for $s>64$,
\begin{equation*}
\frac{d n_2(s)}{ds}=2\mu^{\frac{1}{2}}u_0(\mu).
\end{equation*}
Furthermore, it can be checked easily that $u_0$ is a holomorphic solution around $\mu=0$ of the third-order differential equation 
\begin{equation*}
\mu^2(64\mu-1)\frac{d^3u}{d\mu^3}+\mu(288\mu-3)\frac{d^2u}{d\mu^2}+(208\mu-1)\frac{du}{d\mu}+8u=0,
\end{equation*}
called the \textit{Picard-Fuchs equation} of $X_{s(\mu)}$.
Since the order of the Picard-Fuchs equation equals the rank of $T(X)$, the generic Picard number of $X_s$ must be $19$, and we have from Morrison's result \cite[Cor.~6.4]{Morrison} that $X_s$ admits a Shida-Inose structure for every nonzero $s$. Roughly speaking, this means that there are isogenous elliptic curves $E_s$ and $E_s'$ together with the following diagram:
\begin{center}
\begin{tikzpicture}
\draw [dashed,->](-2,1) -- (-0.5,0) ;
\draw [dashed,->] (2,1) -- (0.5,0) ;
\end{tikzpicture}
\put(-125,32.5){$X_s$}
\put(-11,32.5){$E_s\times E_s'$}
\put(-91,-15){$\Km(E_s\times E_s')$}
\end{center}
Here $\Km(E_s\times E_s')$ is the Kummer surface for $E_s$ and $E_s'$, and the dashed arrows denote rational maps of degree $2$. In addition, $E_s$ is a CM elliptic curve if and only if $X_s$ is singular.
It is known from the results due to Ahlgren, Ono, and Penniston \cite{AOP} and Long \cite{LongII,LongIII} that $u_0\left(-\frac{\mu}{64}\right)$ is a holomorphic solution around $\mu=0$ of the Picard-Fuchs equation of the family of $K3$ surfaces given by the equation
\begin{equation*}
\tilde{X}_\mu : z^2=xy(x+1)(y+1)(x+\mu y).
\end{equation*}
In particular, they proved that the family of elliptic curves associated to $\tilde{X}_\mu$ via a Shioda-Inose structure is 
\begin{equation*}
\tilde{E}_\mu : y^2=(x-1)\left(x^2-\frac{1}{1+\mu}\right).
\end{equation*}
Hence, by simple reparametrization, the family of elliptic curves 
\begin{equation*}
E_s : y^2=(x-1)\left(x^2-\frac{s}{s-64}\right),
\end{equation*}
gives rise to the Shioda-Inose structure of $X_s$, and the $j$-function of $E_s$ is 
\begin{equation*}
j(E_s)=\frac{(s-16)^3}{s}.
\end{equation*}
Recall from \cite[\S A.3]{Silverman} that if $E_s$ is defined over $\mathbb{Q}$, then $E_s$ has complex multiplication if and only if \begin{align*}
j(E_s)\in &\{-640320^3, -5280^3, -960^3, -3\cdot 160^3, -96^3, -32^3, -15^3, \\&\quad 0, 12^3, 20^3, 2\cdot 30^3, 66^3, 255^3\}=:\mathcal{C}_1.
\end{align*}
Furthermore, with the aid of \texttt{Sage}, we find that the set of the CM $j$-invariants in $\mathbb{Q}(\sqrt{2})$ is 
\begin{align*}
\mathcal{C}_1\cup &\{41113158120\pm 29071392966\sqrt{2}, 26125000\pm 18473000\sqrt{2}, 2417472\pm 1707264\sqrt{2},\\ & 3147421320000 \pm 2225561184000\sqrt{2}\}=:\mathcal{C}_2.
\end{align*}
As a consequence, we can explicitly determine the values of $s$ such that $E_s$ has a CM $j$-invariant in $\mathcal{C}_2$. Some of these values are given below, together with $j(E_s)$, the discriminant $D$, and the conductor $f$ of the order of the complex multiplication.  
\begin{table}[ht]
\centering
    \begin{tabular}{ | c | c | c | c |}
    \hline
    $s$ & $j(E_s)$ & $D$ & $f$ \\ \hline
    $16$ & $0$ & $-3$ & $1$ \\ \hline
    $256,-104\pm60\sqrt{3}$ & $2\cdot 30^3$ & $-3$ & $2$ \\ \hline
    $-8,64$ & $12^3$ & $-4$ & $1$ \\ \hline
    $-512,280\pm 198\sqrt{2}$ & $66^3$ & $-4$ & $2$ \\ \hline
    $1,\frac{47\pm 45\sqrt{-7}}{2}$ & $-15^3$ & $-7$ & $1$ \\ \hline
    $4096,-2024\pm 765\sqrt{7}$ & $-15^3$ & $-7$ & $2$ \\ \hline
    $-64,56\pm 40\sqrt{2}$ & $20^3$ & $-8$ & $1$ \\ \hline
    $-1088\pm 768\sqrt{2}$ & $2417472\mp 1707264\sqrt{2}$ & $-24$ & $1$ \\ \hline
    $568+384\sqrt{2}\pm 336\sqrt{3} \pm 216\sqrt{6}$ & $2417472+1707264\sqrt{2}$ & $-24$ & $1$ \\ \hline
    $568\pm 384\sqrt{2}+ 336\sqrt{3} \pm 216\sqrt{6}$ & $2417472-1707264\sqrt{2}$ & $-24$ & $1$ \\ \hline
    \end{tabular}
    \caption{Some values of $s$ for which $E_s$ is CM.}
\label{tab:CM1}
\end{table}

For each value of $s$ in Table~\ref{tab:CM1}, it turns out that $n_2(s)$ (conjecturally) equals rational linear combinations of $L$-values of CM weight three newforms and those of Dirichlet characters, as listed in Table~\ref{tab:f2}. Note, however, that there are several algebraic values of $s$ other than those in Table~\ref{tab:CM1} which yield CM elliptic curves $E_s$, but we have not been able to determine whether the corresponding $n_2(s)$ are related to $L$-values. For example, if $s=16+1600\sqrt[3]{2}-1280\sqrt[3]{4}$, then $j(E_s)=-3\cdot 160^3$, so $E_s$ is CM. We hypothesize from the known examples that $n_2(s)$ should involve exactly three modular $L$-values, though no such conjectural formula has been found. 

Now let us consider the family $Z_s$ of quartic surfaces defined by $C_s=0$. It again follows from Long's result \cite[\S 5.2]{LongI} that if we parameterize $s$ by 
\begin{equation*}
s=s(u):=-\frac{2^{10}u^4}{(u^4-1)^2},
\end{equation*}
then a family of elliptic curves $G_{s(u)}$ whose $j$-function is given by 
\begin{equation*}
j(G_{s(u)})=\frac{64(3u^2+1)^3(u^2+3)^3}{(u^4-1)^2(u^2-1)^2}
\end{equation*}
gives rise to a Shioda-Inose structure of $Z_{s(u)}.$ Indeed, a Weierstrass form of $G_s$ is explicitly determined in our forthcoming paper \cite{SamartII}. Thus it can be shown in a similar manner that if $s$ is an algebraic number in the second column of Table~\ref{tab:f4}, then $Z_s$ is a singular $K3$ surface, and $n_4(s)$ relates to modular and Dirichlet $L$-values. See below for a table containing information analogous to that in Table~\ref{tab:CM1}. 

What is remarkable about a singular $K3$ surface defined over $\mathbb{Q}$ is that it is always modular, as mentioned in the introduction of this paper. Nevertheless, the modularity of singular $K3$ surfaces defined over arbitrary number fields is not known. The numerical evidences of relationships between the three-variable Mahler measures and $L$-values obtained in Section~\ref{Sec:six} might give us some clues about modularity of the corresponding $K3$ surfaces defined over some number fields. Nevertheless, this certainly requires further investigation. It would also be highly desirable to find all possible Mahler measure formulas $n_j(s),j=1,2,3$ which are expressible in terms of special $L$-values. 
\begin{table}[ht]
\centering
    \begin{tabular}{ | c | c | c | c |}
    \hline
    $s$ & $j(G_s)$ & $D$ & $f$ \\ \hline
    $-144,26856-15300\sqrt{3}$ & $2\cdot 30^3$ & $-3$ & $2$ \\ \hline
    $26856+15300\sqrt{3}$ & $1417905000+818626500\sqrt{3}$ & $-3$ & $4$ \\ \hline
    $648,143208-101574\sqrt{2}$ & $66^3$ & $-4$ & $2$ \\ \hline
    $-12288$ & $76771008+44330496\sqrt{3}$ & $-4$ & $3$ \\ \hline
    $143208+101574\sqrt{2}$ & $41113158120+29071392966\sqrt{2}$ & $-4$ & $4$ \\ \hline
    $81$ & $-15^3$ & $-7$ & $1$ \\ \hline
    $-3969,8292456-3132675\sqrt{7}$ & $255^3$ & $-7$ & $2$ \\ \hline
    $8292456+3132675\sqrt{7}$ & $137458661985000+51954490735875\sqrt{7}$ & $-7$ & $4$ \\ \hline
    $256,3656-2600\sqrt{2}$ & $20^3$ & $-8$ & $1$ \\ \hline
    $3656+2600\sqrt{2}$ & $26125000+18473000\sqrt{2}$ & $-8$ & $2$ \\ \hline
    $614656$ & $188837384000+77092288000\sqrt{6}$ & $-8$ & $3$ \\ \hline
    $\frac{-192303\pm 85995\sqrt{5}}{2}$ & $\frac{37018076625\mp 16554983445\sqrt{5}}{2}$ & $-15$ & $2$ \\ \hline
    $-1024$ & $632000+282880\sqrt{5}$ & $-20$ & $1$ \\ \hline
    $2304,1207368+853632\sqrt{2}-$ & $2417472+1707264\sqrt{2}$ & $-24$ & $1$ \\ 
    $697680\sqrt{3}-493272\sqrt{6}$ & & & \\ \hline
    $1207368-853632\sqrt{2}-$ & $2417472-1707264\sqrt{2}$ & $-24$ & $1$ \\ 
    $697680\sqrt{3}+493272\sqrt{6}$ & & & \\ \hline
    $1207368\pm 853632\sqrt{2}+$ & $5835036074184\pm 4125993565824\sqrt{2}+$ & $-24$ & $2$ \\ 
    $697680\sqrt{3}\pm 493272\sqrt{6}$ &$3368859648336\sqrt{3}\pm 2382143496408\sqrt{6}$ & & \\ \hline
    $20736$ & $212846400+95178240\sqrt{5}$ & $-40$ & $1$ \\ \hline
    $-82944$ & $3448440000+956448000\sqrt{13}$ & $-52$ & $1$ \\ \hline
    $-893952\pm 516096\sqrt{3}$ & $799200236736\mp 461418467328\sqrt{3}+$ & $-84$ & $1$ \\ 
     &$302069634048\sqrt{7}\mp 174399982848\sqrt{21}$ & & \\ \hline
    $347648256\pm 141926400\sqrt{6}$ & $120858928019208000\pm 49340450750976000\sqrt{6}\pm$ & $-168$ & $1$ \\ 
     &$32300907105600000\sqrt{14}+26373580212672000\sqrt{21}$ & & \\ \hline
    \end{tabular}
    \caption{Some values of $s$ for which $G_s$ is CM.}
\label{tab:CM2}
\end{table}
\section{Functional equations in the three-variable case} \label{Sec:five}
One has seen from \cite{LR} that $m_2(t)$ satisfies some functional equations, which can be applied in establishing new Mahler measure formulas as shown in Section~\ref{Sec:two}. This section aims to derive a functional equation for three-variable Mahler measures. We will show that
\begin{theorem}\label{T:FE1}
If $t\in\mathbb{C}\backslash \{0\}$ and $|t|$ is sufficiently small, then 
\begin{align*}
n_2\left(\frac{16}{t(1-t)}\right)&=9n_2\left(\frac{4(1+\sqrt{1-t})^6}{t^2\sqrt{1-t}}\right)+4n_2\left(\frac{-2^{10}(1+\sqrt{1-t})^6\sqrt{1-t}}{t^4}\right)\\ &\qquad\qquad -n_2\left(\frac{-16(1-t)^2}{t}\right) -8n_2\left(\frac{2(1+\sqrt[4]{1-t})^{12}}{t(1-\sqrt{1-t})^3\sqrt[4]{1-t}}\right).
\end{align*}
\end{theorem}
\begin{proof}
The proof requires some preliminary results from \cite[Thm.~2.3]{RogersMain} and Ramanujan's theory of elliptic functions. 
Following notations in \cite{RogersMain}, we let 
\begin{equation*}
G(q):=\Re\left(-\log(q)+240\sum_{n=1}^{\infty}n^2\log(1-q^n)\right), \quad
\chi(q):=\displaystyle\prod_{n=0}^{\infty}\left(1+q^{2n+1}\right).
\end{equation*}
Recall from Rogers' result that if $|q|$ is sufficiently small, then the following matrix equation holds:
\begin{equation*}
\left( \begin{array}{c} G(q)\\ G(-q) \\ G(q^2) \end{array} \right)=
\begin{pmatrix}
-19 & -4  & 12 \\
-4  & -19 & 12 \\
-3  & -3  & 4 
\end{pmatrix}
\left( \begin{array}{c} n_2(s_2(q))\\ n_2(s_2(-q)) \\ 2n_2\left(s_2\left(q^2\right)\right)-n_2\left(s_2\left(-q^2\right)\right) \end{array} \right).
\end{equation*}
Expressing $G(q^2)$ in two different ways, one finds that 
\begin{equation}\label{E:FE1}
n_2(s_2(q))=9n_2\left(s_2\left(q^2\right)\right)+4n_2\left(s_2\left(-q^4\right)\right)-n_2(s_2(-q))-8n_2\left(s_2\left(q^4\right)\right).
\end{equation}
Now let 
\begin{equation*}
z_2(t)=\pFq{2}{1}{\frac{1}{2},\frac{1}{2}}{1}{t},\qquad y_2(t)=\frac{\pi z_2(1-t)}{z_2(t)},\qquad q_2(t)=e^{-y_2}.
\end{equation*}
Note that $q_2(t)$ defined above is sometimes called the \textit{signature $2$ elliptic nome}. It is known from \cite[\S 17]{Berndt} that the following identities hold:
\begin{align*}
\chi(q_2)&=2^{1/6}\left(\frac{q_2}{t(1-t)}\right)^{1/{24}},\qquad \chi(-q_2)=2^{1/6}(1-t)^{1/{12}}\left(\frac{q_2}{t}\right)^{1/{24}},\\
\chi(-q_2^2)&=2^{1/3}(1-t)^{1/{24}}\left(\frac{q_2}{t}\right)^{1/{12}}.
\end{align*}
Moreover, we can deduce formulas for $\chi(q_2^2),\chi(q_2^4),$ and $\chi(-q_2^4)$ from the identities above using a process called \textit{obtaining a formula by duplication}; that is, if we have $\Omega(t,q_2,z_2)=0$, then  $$\Omega\left(\left(\frac{1-\sqrt{1-t}}{1+\sqrt{1+t}}\right)^2,q_2^2,\frac{z_2(1+\sqrt{1-t})}{2}\right)=0.$$
Therefore, by some manipulation, we find that 
\begin{align*}
\chi^{24}(q_2^2)&=\frac{4(1+\sqrt{1-t})^{6}}{t^2\sqrt{1-t}}q_2^2,\qquad \chi^{24}(q_2^4)=\frac{2(1+\sqrt[4]{1-t})^{12}}{t(1-\sqrt{1-t})^3\sqrt[4]{1-t}}q_2^4,\\
\chi^{24}(-q_2^4)&=\frac{2^{10}(1+\sqrt{1-t})^6\sqrt{1-t}}{t^4}q_2^4.
\end{align*}
The theorem then follows immediately from these identities and \eqref{E:FE1}.
\end{proof}
As an application of Theorem~\ref{T:FE1}, we can deduce a five-term relation 
\begin{align*}
n_2(64)&=9n_2\left(280+198\sqrt{2}\right)+4n_2\left(-143360-101376\sqrt{2}\right)\\&\qquad\qquad-n_2(-8)-8n_2\left(71704+50688\sqrt{2}+60282\sqrt[4]{2}+42633\sqrt[4]{8}\right)
\end{align*}
by letting $t=1/2$. It would be interesting to see if each term in the equation above is related to special $L$-values. It turns out that only a partial answer can be given here. First, it was rigorously proved in \cite[Thm.~1.2]{Samart} that $n_2(64)=8L'(g_{16},0),$ where $g_{16}(\tau)=\eta^6(4\tau)\in S_3(\Gamma_0(16),\chi_{-4}).$ Then, using the hypergeometric representation of $n_2(s)$ given in \cite[Prop.~2.2]{RogersMain}, we are able to verify numerically that the following formulas hold:
\begin{align*}
n_2(-8)&\stackrel{?}=4L'(g_{16},0)+L'(\chi_{-4},-1),\\
n_2\left(280+198\sqrt{2}\right)&\stackrel{?}=\frac{1}{8}(36L'(g_{16},0)+4L'(g_{64},0)+13L'(\chi_{-4},-1)+4L'(\chi_{-8},-1)),
\end{align*}
where $g_{64}(\tau)$ is the normalized newform of weight $3$ and level $64$ with rational Fourier coefficients. Nevertheless, no similar evidence for the remaining two terms has been found. From the previous examples and numerical observations exhibited at the end of this paper, it is not unreasonable to conjecture that
\begin{center}
$n_2\left(-143360-101376\sqrt{2}\right)$ and $n_2\left(71704+50688\sqrt{2}+60282\sqrt[4]{2}+42633\sqrt[4]{8}\right)$
\end{center}
involve two and four modular $L$-values, respectively, corresponding to weight $3$ newforms of higher level. However, we are still unable to find the $L$-values of the newforms that are likely to be our possible candidates.

It is also possible to obtain a functional equation for $n_4(s)$ defined in Section~\ref{Sec:three} using similar arguments above. Again, we see from \cite{RogersMain} that for $|q|$ sufficiently small 
\begin{equation*}
\left( \begin{array}{c} G(q)\\ G(-q) \\ G(q^2) \end{array} \right)=
\begin{pmatrix}
-5 & -2  & 4 \\
-2  & -5 & 4 \\
-1  & -1 & 2 
\end{pmatrix}
\left( \begin{array}{c} n_4(s_4(q))\\ n_4(s_4(-q)) \\ n_4(s_4(q^2)) \end{array} \right).
\end{equation*}
Hence we find that 
\begin{equation*}
n_4(s_4(q))=7n_4\left(s_4\left(q^2\right)\right)+2n_4\left(s_4\left(-q^2\right)\right)-n_4(s_4(-q))-4n_4\left(s_4\left(q^4\right)\right).
\end{equation*}
To express $s_4(q),s_4(-q),s_4(q^2),s_4(-q^2),$ and $s_4(q^4)$ in terms of algebraic functions of some parameter we need the Ramanujan's theory of signature $4$. (See \cite{BBG} for references.) However, the results we found are quite complicated because of multiple radical terms, so we do not include them here.

\section{Conjectural formulas of three-variable Mahler measures} \label{Sec:six}
We end this paper by tabulating all three-variable Mahler measure formulas that we found from numerical computations. The references to the proved formulas are given in the last column of each table. In Table~\ref{tab:f2}-\ref{tab:f4}, we use the following shorthand notations:
\begin{equation*}
d_k:=L'(\chi_{-k},-1), \qquad M_N := L'(g_N,0), \qquad M_{N\otimes D} := L'(g_N\otimes \chi_{D},0),
\end{equation*}
where $g_N$ is a normalized newform with rational Fourier coefficients in $S_3(\Gamma_0(N),\chi_{-N})$, and $g_N\otimes \chi_{D}$ is the quadratic twist of $g_N$ by $\chi_{D}$. If there are more than one such newforms, we shall distinguish them using superscripts. Each value of $\tau$ in the first column of each table can be determined as follows: Recall from the proof of \cite[Thm.2.3]{RogersMain} that if 
\begin{equation*}
q_j(\alpha)=\exp\left(-\frac{\pi}{\sin(\pi/j)}\frac{\pFq{2}{1}{\frac{1}{j},\frac{j-1}{j}}{1}{1-\alpha}}{\pFq{2}{1}{\frac{1}{j},\frac{j-1}{j}}{1}{\alpha}}\right),
\end{equation*}
then $s_2(q_2(\alpha))=\frac{16}{\alpha(1-\alpha)},s_3(q_3(\alpha))=\frac{27}{\alpha(1-\alpha)},$ and $s_4(q_4(\alpha))=\frac{64}{\alpha(1-\alpha)}.$ Hence we can recover a value of $\tau$ corresponding to $s_j(\tau)$ easily using these relations. For instance, each $\tau$ in Table~\ref{tab:f2} is given by
\begin{equation*}
\tau=\frac{i}{2}\frac{\pFq{2}{1}{\frac{1}{2},\frac{1}{2}}{1}{1-\frac{1+\sqrt{1-\frac{64}{s_2(\tau)}}}{2}}}{\pFq{2}{1}{\frac{1}{2},\frac{1}{2}}{1}{\frac{1+\sqrt{1-\frac{64}{s_2(\tau)}}}{2}}}.
\end{equation*}
%\begin{table}[ht]
%\centering
    \begin{longtable}{ | c | c | c | c |}
    \hline
    $\tau$ & $s_2(\tau)$ & $n_2(s_2(\tau))$ & Reference \\ \hline
    $\frac{\sqrt{-1}}{2}$ & $64$ & $8M_{16}$ & \cite{Samart} \\ \hline
    $\frac{1+\sqrt{-1}}{2}$ & $-8$ & $4M_{16}+d_4$ & - \\ \hline
    $\frac{\sqrt{-4}}{2}$ & $280+198\sqrt{2}$ & $\frac{1}{8}\left(36M_{16}+4M_{16\otimes 8}+13d_4+4d_8\right)$ & - \\ \hline
    $\frac{2+\sqrt{-1}}{4}$ & $280-198\sqrt{2}$ & $\frac{1}{2}\left(36M_{16}-4M_{16\otimes 8}-13d_4+4d_8\right)$ & - \\ \hline        
    $\frac{1+\sqrt{-4}}{2}$ & $-512$ & $M_{64}+d_8$ & - \\ \hline
    $\frac{\sqrt{-2}}{2}$ & $56+40\sqrt{2}$ & $\frac{1}{4}\left(60M_{8}+4M_{8\otimes 8}+4d_4+d_8\right)$ & - \\ \hline
    $\frac{2+\sqrt{-2}}{4}$ & $56-40\sqrt{2}$ & $\frac{1}{2}\left(60M_{8}-4M_{8\otimes 8}+4d_4-d_8\right)$ & - \\ \hline
    $\frac{1+\sqrt{-2}}{2}$ & $-64$ & $2\left(M_{8\otimes 8}+d_4\right)$ & - \\ \hline
    $\frac{\sqrt{-3}}{2}$ & $256$ & $\frac{4}{3}\left(M_{12\otimes (-4)}+2d_4\right)$ & \cite{Samart} \\ \hline
    $\frac{1+\sqrt{-3}}{4}$ & $16$ & $8M_{12}$ & - \\ \hline
    $\frac{3+\sqrt{-3}}{6}$ & $-104+60\sqrt{3}$ & $\frac{1}{2}\left(4M_{12\otimes (-4)}-36M_{12}+15d_3-8d_4\right)$ & - \\ \hline
    $\frac{1+\sqrt{-3}}{2}$ & $-104-60\sqrt{3}$ & $\frac{1}{6}\left(4M_{12\otimes (-4)}+36M_{12}+15d_3+8d_4\right)$ & - \\ \hline   
    $\frac{\sqrt{-6}}{2}$ & $568+384\sqrt{2}$ & $\frac{1}{24}\bigl(60M_{24}^{(1)}+12M_{24}^{(2)}+4M_{24\otimes (-8)}^{(1)}+4M_{24\otimes (-8)}^{(2)}$ & - \\ 
                            & $+336\sqrt{3}+216\sqrt{6}$ & $+60d_3+24d_4+8d_8+d_{24}\bigr)$ &  \\ \hline
    $\frac{6+\sqrt{-6}}{12}$ & $568+384\sqrt{2}$ & $\frac{1}{4}\bigl(60M_{24}^{(1)}+12M_{24}^{(2)}-4M_{24\otimes (-8)}^{(1)}-4M_{24\otimes (-8)}^{(2)}$ & - \\ 
                            & $-336\sqrt{3}-216\sqrt{6}$ & $-60d_3+24d_4+8d_8-d_{24}\bigr)$ &  \\ \hline
    $\frac{\sqrt{-6}}{6}$ & $568-384\sqrt{2}$ & $\frac{1}{12}\bigl(60M_{24}^{(1)}-12M_{24}^{(2)}+4M_{24\otimes (-8)}^{(1)}-4M_{24\otimes (-8)}^{(2)}$ & - \\ 
                            & $+336\sqrt{3}-216\sqrt{6}$ & $+60d_3+24d_4-8d_8-d_{24}\bigr)$ &  \\ \hline
    $\frac{-2+\sqrt{-6}}{10}$ & $568-384\sqrt{2}$ & $\frac{1}{12}\bigl(60M_{24}^{(1)}-12M_{24}^{(2)}-4M_{24\otimes (-8)}^{(1)}+4M_{24\otimes (-8)}^{(2)}$ & - \\ 
                            & $-336\sqrt{3}+216\sqrt{6}$ & $+60d_3-24d_4+8d_8-d_{24}\bigr)$ &  \\ \hline                              
    $\frac{3+\sqrt{-6}}{6}$ & $-1088+768\sqrt{2}$ & $M_{24\otimes (-8)}^{(1)}-M_{24\otimes (-8)}^{(2)}-6d_4+2d_8$ & - \\ \hline
    $\frac{1+\sqrt{-6}}{2}$ & $-1088-768\sqrt{2}$ & $\frac{1}{3}\left(M_{24\otimes (-8)}^{(1)}+M_{24\otimes (-8)}^{(2)}+6d_4+2d_8\right)$ & - \\ \hline
    $\frac{\sqrt{-7}}{2}$ & $4096$ & $\frac{4}{7}\left(M_{7\otimes (-4)}+8d_4\right)$ & - \\ \hline 
     $\frac{3+\sqrt{-7}}{8}$ & $1$ & $8M_{7}$ & - \\ \hline     
     $\frac{\pm 1+\sqrt{-7}}{8}$ & $\displaystyle\frac{47\pm 45\sqrt{-7}}{2}$ & $\frac{4}{7}\left(54M_{7}+d7\right)$ & - \\ \hline       
    $\frac{7+\sqrt{-7}}{14}$ & $-2024+765\sqrt{7}$ & $\frac{1}{2}\left(4M_{7\otimes (-4)}-384M_{7}-32d_4+11d_7\right)$ & - \\ \hline  
    $\frac{1+\sqrt{-7}}{2}$ & $-2024-765\sqrt{7}$ & $\frac{1}{14}\left(4M_{7\otimes (-4)}+384M_{7}+32d_4+11d_7\right)$ & - \\ \hline           
%    \end{tabular}
    \caption{Some $L$-value expressions of $n_2(s)$}
\label{tab:f2}
\end{longtable}

%\begin{table}[p]
%\centering
    \begin{longtable}{ | c | c | c | c |}
    \hline
    $\tau$ & $s_3(\tau)$ & $n_3(s_3(\tau))$ & Reference \\ \hline
    $\frac{1+\sqrt{-2}}{3}$ & $8$ & $15M_{8}$ & - \\ \hline
    $\frac{\sqrt{-3}}{3}$ & $108$ & $15M_{12}$ & \cite{RogersMain} \\ \hline
    $\frac{\sqrt{-6}}{3}$ & $216$ & $\frac{15}{4}\left(M_{24}^{(2)}+d_3\right)$ & \cite{Samart} \\ \hline
    $\frac{\sqrt{-9}}{3}$ & $288+168\sqrt{3}$ & $\frac{5}{12}\left(3M_{36}^{(2)}+3M_{36}^{(1)}+6d_3+4d_4\right)$ & - \\ \hline
    $\frac{1+\sqrt{-1}}{2}$ & $288-168\sqrt{3}$ & $\frac{5}{6}\left(3M_{36}^{(2)}-3M_{36}^{(1)}-6d_3+4d_4\right)$ & - \\ \hline
    $\frac{\sqrt{-12}}{3}$ & $1458$ & $\frac{15}{8}\left(9M_{12}+2d_4\right)$ & \cite{Samart} \\ \hline 
    $\frac{\sqrt{-15}}{3}$ & $3375$ & $\frac{3}{5}\left(20M_{15}^{(2)}+13d_3\right)$ & - \\ \hline
    $\frac{\sqrt{-18}}{3}$ & $3704+1456\sqrt{6}$ & $\frac{5}{24}\left(3M_{8\otimes (-3)}+72M_{8}+18d_3+4d_8\right)$ & - \\ \hline   
    $\frac{\sqrt{-2}}{2}$ & $3704-1456\sqrt{6}$ & $\frac{5}{12}\left(3M_{8\otimes (-3)}-72M_{8}-18d_3+4d_8\right)$ & - \\ \hline
    $\frac{\sqrt{-21}}{3}$ & $7344+2808\sqrt{7}$ & $\frac{15}{28}\left(M_{84}^{(2)}+M_{84}^{(4)}+4d_4+2d_7\right)$ & - \\ \hline         $\frac{3+\sqrt{-21}}{6}$ & $7344-2808\sqrt{7}$ & $\frac{15}{14}\left(M_{84}^{(2)}-M_{84}^{(4)}-4d_4+2d_7\right)$ & - \\ \hline
    $\frac{\sqrt{-24}}{3}$ & $14310+8262\sqrt{3}$ & $\frac{15}{32}\left(7M_{24}^{(2)}+M_{24\otimes (-8)}^{(2)}+11d_3+6d_4\right)$ & - \\ \hline
    $\frac{-3+\sqrt{-6}}{2}$ & $14310-8262\sqrt{3}$ & $\frac{15}{8}\left(7M_{24}^{(2)}-M_{24\otimes (-8)}^{(2)}+11d_3-6d_4\right)$ & - \\ \hline
    $\frac{\sqrt{-30}}{3}$ & $48168+15120\sqrt{10}$ & $\frac{3}{40}\left(5M_{120}^{(2)}+5M_{120}^{(4)}+5d_{15}+2d_{24}\right)$ & - \\ \hline
    $\frac{6+\sqrt{-30}}{6}$ & $48168-15120\sqrt{10}$ & $\frac{3}{20}\left(5M_{120}^{(2)}-5M_{120}^{(4)}+5d_{15}-2d_{24}\right)$ & - \\ \hline
%    \end{tabular}
    \caption{Some $L$-value expressions of $n_3(s)$}
\label{tab:f3}
\end{longtable}

%\begin{table}[h]
%\centering
    \begin{longtable}{ | c | c | c | c |}
    \hline
    $\tau$ & $s_4(\tau)$ & $n_4(s_4(\tau))$ & Reference \\ \hline
    $\frac{\sqrt{-2}}{2}$ & $256$ & $40M_8$ & \cite{RogersMain} \\ \hline
    $\frac{\sqrt{-8}}{2}$ & $3656+2600\sqrt{2}$ & $\frac{5}{8}\left(4M_{8\otimes 8}+28M_{8}+4d_4+d_8\right)$ & - \\ \hline
    $\frac{1+\sqrt{-2}}{2}$ & $3656-2600\sqrt{2}$ & $\frac{5}{4}\left(4M_{8\otimes 8}-28M_8+4d_4-d_8\right)$ & - \\ \hline
    $\frac{\sqrt{-12}}{2}$ & $26856+15300\sqrt{3}$ & $\frac{5}{12}\left(4M_{12\otimes (-4)}+20M_{12}+11d_3+8d_4\right)$ & Thm.~\ref{P:threevar} \\ \hline
    $\frac{\sqrt{-3}}{2}$ & $26856-15300\sqrt{3}$ & $\frac{5}{6}\left(4M_{12\otimes (-4)}-20M_{12}-11d_3+8d_4\right)$ & Thm.~\ref{P:threevar} \\ \hline
    $\frac{1+\sqrt{-3}}{2}$ & $-144$ & $\frac{10}{3}\left(4M_{12}+d_3\right)$ & - \\ \hline
    $\frac{\sqrt{-4}}{2}$ & $648$ & $\frac{5}{2}\left(4M_{16}+d_4\right)$ & \cite{Samart} \\ \hline
    $\frac{\sqrt{-16}}{2}$ & $143208+101574\sqrt{2}$ & $\frac{5}{16}\left(4M_{16\otimes 8}+20M_{16}+9d_4+4d_8\right)$ & - \\ \hline
    $\frac{1+\sqrt{-4}}{2}$ & $143208-101574\sqrt{2}$ & $\frac{5}{8}\left(4M_{16\otimes 8}-20M_{16}-9d_4+4d_8\right)$ & - \\ \hline
    $\frac{1+\sqrt{-5}}{2}$ & $-1024$ & $\frac{8}{5}\left(5M_{20}^{(1)}+2d_4\right)$ & - \\ \hline
    $\frac{\sqrt{-6}}{2}$ & $2304$ & $\frac{20}{3}\left(M_{24}^{(1)}+d_3\right)$ & \cite{Samart} \\ \hline
    $\frac{\sqrt{-24}}{2}$ & $1207368+853632\sqrt{2}$ & $\frac{5}{48}\bigl(4M_{24\otimes (-8)}^{(1)}+4M_{24\otimes (-8)}^{(2)}+28M_{24}^{(1)}+12M_{24}^{(2)}$ & - \\ 
                            & $+697680\sqrt{3}+493272\sqrt{6}$ & $+28d_3+24d_4+8d_8+d_{24}\bigr)$ &  \\ \hline
    $\frac{1+\sqrt{-6}}{2}$ & $1207368+853632\sqrt{2}$ & $\frac{5}{24}\bigl(4M_{24\otimes (-8)}^{(1)}+4M_{24\otimes (-8)}^{(2)}-28M_{24}^{(1)}-12M_{24}^{(2)}$ & - \\ 
                            & $-697680\sqrt{3}-493272\sqrt{6}$ & $-28d_3+24d_4+8d_8-d_{24}\bigr)$ &  \\ \hline
    $\frac{\sqrt{-6}}{4}$ & $1207368-853632\sqrt{2}$ & $\frac{5}{16}\bigl(4M_{24\otimes (-8)}^{(1)}-4M_{24\otimes (-8)}^{(2)}+28M_{24}^{(1)}-12M_{24}^{(2)}$ & - \\ 
                            & $+697680\sqrt{3}-493272\sqrt{6}$ & $-28d_3-24d_4+8d_8+d_{24}\bigr)$ &  \\ \hline
    $\frac{2+\sqrt{-6}}{4}$ & $1207368-853632\sqrt{2}$ & $\frac{5}{12}\bigl(-4M_{24\otimes (-8)}^{(1)}+4M_{24\otimes (-8)}^{(2)}+28M_{24}^{(1)}-12M_{24}^{(2)}$ & - \\ 
                            & $-697680\sqrt{3}+493272\sqrt{6}$ & $+28d_3-24d_4+8d_8-d_{24}\bigr)$ &  \\ \hline
    $\frac{\sqrt{-28}}{2}$ & $8292456+3132675\sqrt{7}$ & $\frac{5}{28}\left(4M_{7\otimes (-4)}+224M_{7}+32d_4+7d_7\right)$ & - \\ \hline
    $\frac{\sqrt{-7}}{2}$ & $8292456-3132675\sqrt{7}$ & $\frac{5}{14}\left(4M_{7\otimes (-4)}-224M_{7}+32d_4-7d_7\right)$ & - \\ \hline 
    $\frac{\sqrt{14}+\sqrt{-28}}{8}$ & $81$ & $40M_{7}$ & - \\ \hline
    $\frac{1+\sqrt{-7}}{2}$ & $-3969$ & $\frac{10}{7}\left(40M_{7}+d_7\right)$ & - \\ \hline
    $\frac{1+\sqrt{-9}}{2}$ & $-12288$ & $\frac{40}{9}\left(M_{36}^{(1)}+2d_3\right)$ & - \\ \hline
    $\frac{\sqrt{-10}}{2}$ & $20736$ & $\frac{4}{5}\left(5M_{40}^{(1)}+2d_8\right)$ & \cite{Samart} \\ \hline
    $\frac{1+\sqrt{-13}}{2}$ & $-82944$ & $\frac{40}{13}\left(M_{52}^{(1)}+2d_4\right)$ & - \\ \hline  
    $\frac{3+\sqrt{-15}}{6}$ & $\displaystyle\frac{-192303+85995\sqrt{5}}{2}$ & $\frac{1}{5}\left(160M_{15}^{(1)}-120M_{15}^{(2)}-88d_3+5d_{15}\right)$ & - \\ \hline
    $\frac{1+\sqrt{-15}}{2}$ & $\displaystyle\frac{-192303-85995\sqrt{5}}{2}$ & $\frac{1}{15}\left(160M_{15}^{(1)}+120M_{15}^{(2)}+88d_3+5d_{15}\right)$ & - \\ \hline   
    $\frac{\sqrt{-18}}{2}$ & $614656$ & $\frac{40}{3}(5M_{8}+d_3)$ & \cite{Samart} \\ \hline
    $\frac{3+\sqrt{-21}}{6}$ & $-893952+516096\sqrt{3}$ & $\frac{20}{7}\left(M_{84}^{(3)}-M_{84}^{(4)}+8d_3-4d_4\right)$ & - \\ \hline
    $\frac{1+\sqrt{-21}}{2}$ & $-893952-516096\sqrt{3}$ & $\frac{20}{21}\left(M_{84}^{(3)}+M_{84}^{(4)}+8d_3+4d_4\right)$ & - \\ \hline    
    $\frac{\sqrt{-42}}{42}$ & $347648256+141926400\sqrt{6}$ & $\frac{10}{21}\left(M_{168}^{(3)}+M_{168}^{(4)}+20d_3+4d_8\right)$ & - \\ 
\hline
$\frac{\sqrt{-42}}{14}$ & $347648256-141926400\sqrt{6}$ & $\frac{10}{7}\left(M_{168}^{(3)}-M_{168}^{(4)}-20d_3+4d_8\right)$ & - \\
     \hline     
%    \end{tabular}
    \caption{Some $L$-value expressions of $n_4(s)$}
\label{tab:f4}
\end{longtable}

\bibliographystyle{amsplain}

\end{document}